\documentclass[12pt]{extarticle}
\usepackage{amsmath, amsthm, amssymb, graphicx, svg, url, xcolor}
\usepackage{hyperref}
\hypersetup{
colorlinks=true,
linkcolor=blue,
filecolor=magenta,
urlcolor=red,
citecolor=gray,
}
\usepackage[capitalise]{cleveref}
\graphicspath{{./Images}}

\newtheorem{theorem}{Theorem}[section]
\newtheorem{corollary}{Corollary}[section]

\title{Uniform Shared Neighborhood Structures in Edge-Regular Graphs}
\author{
        Jared DeLeo\footnote{Auburn University, jmd0150@auburn.edu}\\
        Department of Mathematics and Statistics\\
        Auburn University\\
        Auburn, AL 36830
        }
\date{August 28, 2024}

\begin{document}

\maketitle
\sloppy

{\sc Abstract:} A shared neighborhood structure (SNS) in a graph is a subgraph induced by the intersection of the open neighbor sets of two adjacent vertices. If a SNS is the same for all adjacent vertices in an edge-regular graph, call the SNS a uniform shared neighborhood structure (USNS). USNS-forbidden graphs (graphs which cannot be a USNS of an edge-regular graph) and USNS in graph products of edge-regular graphs are examined.\\


\section{Preliminaries}

Let $G = (V,E)$ be a finite, simple \textit{graph} with vertex set $V = V(G)$ and edge set $E = E(G)$. If $uv \in E(G)$ for vertices $u,v \in V(G)$, then their adjacency is denoted $u \sim v$. The \textit{degree} of a vertex is the number of edges it is incident to. Because $G$ is simple, the degree of $v \in V(G)$ is also the number of vertices it is adjacent to. A graph $G$ is \textit{regular} if the degrees of the vertices in $V(G)$ are all the same. The \textit{open neighborhood} of a vertex $u$ in $G$, denoted $N_G(u)$, is the set of vertices $u$ is adjacent to. If $G$ is understood, this open neighborhood will be denoted $N(u)$. A graph $G$ is \textit{edge-regular} if $G$ is both regular and, for some $\lambda$, every pair of adjacent vertices in $G$ have exactly $\lambda$ common (or shared) neighbors. If $G$ is edge-regular, we say $G \in ER(n,d,\lambda)$, where $|V(G)| = n$, $G$ is regular of degree $d$, and $|N(u) \cap N(v)| = \lambda$ for all $uv \in E(G)$.

An \textit{induced subgraph} of $G$ is a graph $H$ such that $V(H) \subseteq V(G)$, $E(H)$ contains all of the edges of $G$ among the vertices of $V(H)$, and only those edges. The induced subgraph $H$ of $G$ is denoted as $G[V(H)]$. If $G[N_G(u) \cap N_G(v)] \cong H$ for all $u \sim v; u,v \in V(G)$, where $\cong$ denotes a graph isomorphism, then $G$ has a \textit{uniform shared neighborhood structure}, abbreviated \textit{USNS}. For instance, letting $K_n$ denote the complete graph on $n$ vertices, $G = K_3 \in ER(3,2,1)$ has USNS $K_1$. 

For graphs $G$ and $H$, define $G + H$ to be the graph formed from $G$ and $H$ where $V(G+H) = V(G) \cup V(H)$ (where $V(G)$ and $V(H)$ are disjoint) and $E(G+H) = E(G) \cup E(H)$. Further, for a graph $G$ and positive integer $m$, define $mG$ to be $m$ disjoint copies of $G$. That is, $mG = G + G + \dots + G$.

Edge-regular graphs do not need to have a USNS. If $G$ is the Cartesian product of $K_4$ and $\{K_6\}\setminus\text{\{a perfect matching in $K_6$\}}$, $G \in ER(24,7,2)$ has two different shared neighborhood structures (SNS): $K_2$ and $2K_1$. Also, a SNS for one pair of adjacent vertices may also be the SNS for a different pair of adjacent vertices. Suppose $G$ is $\{K_6\}\setminus\text{\{a perfect matching in $K_6$\}}$ as in \cref{fig:K6MinusPM}. Then $G \in ER(6,4,2)$ contains a $2K_1$ USNS, of which the pair of vertices in the USNS is shared by two pairs of adjacent vertices.

\begin{figure}[h]
    \centering
    \includegraphics[width=.2\paperwidth]{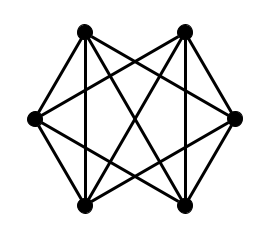}
    \caption{$K_6$ with a perfect matching removed}
    \label{fig:K6MinusPM}
\end{figure}

A number of studies of edge-regular graphs have focused on the parameter $\lambda$. These graphs with $\lambda = 1$ have been studied in \cite{Bragan} and \cite{RCA}, while those with $\lambda = 2$ have been studied in \cite{Glorioso}. Additionally, in \cite{Bragan}, \cite{RCA}, and \cite{Glorioso}, constructions are described for edge-regular graphs.

Outside of specific $\lambda$ values, relations amongst the parameters of an edge-regular graph have also been studied, notably when $d = \lambda + k$ for $k \in \{1,2,3\}$ in \cite{ExtremalER}. The research in \cite{ExSubfamilyExFamilyER} also examines parameter relations, specifically as it pertains to $n$, $\lambda$, and the number of vertices missing from any shared neighborhood.

The research presented in this paper will pertain more to the structure of edge-regular graphs, akin to the research presented in \cite{Neumaier}, which constructs a specific type of edge-regular graph, a \textit{Neumaier} graph. Within the body of this research, there is an emphasis on families of forbidden subgraphs of edge-regular graphs, as well as corresponding constructions of graphs in these families.

\section{Forbidden USNS}

There are families of graphs that cannot be a USNS in any edge-regular graph; call these USNS-forbidden graphs. These results utilize constructions by contradiction. For a graph $G$ and $u,v \in V(G)$, let $A(u,v)$ denote the set of vertices in $G$ that are adjacent to $u$ but not to $v$, and let $B(u,v)$ denote the set of vertices in $G$ that are adjacent to $v$ but not to $u$. Finally, let $X(u,v)$ denote the set of vertices in $G$ that are adjacent neither to $u$ nor $v$.

Let $P_m$ be the path graph on $m$ vertices.

\begin{theorem}\label{thm:P3}
    If $G \in ER(n,d,3)$ with a USNS, then the USNS $\ncong P_3$.
\end{theorem}
\begin{proof}
    By way of contradiction, let $u \sim v$, and let $N(u) \cap N(v) = \{w_1,w_2,w_3\}$, where $G[N(u) \cap N(v)] \cong P_3$. Without loss of generality, let $w_1 \sim w_2 \sim w_3$ and $w_1 \nsim w_3$. Then as $w_1 \sim w_2$, $G[N(w_1) \cap N(w_2)] \cong P_3$. As two of $w_1$ and $w_2$'s common neighbors are $u$ and $v$, there must exist a third vertex, say $z$, such that $N(w_1) \cap N(w_2) = \{u,v,z\}$ and $G[N(w_1) \cap N(w_2)] \cong P_3$.
    
    Without loss of generality, suppose $z \sim u$. Then $\{w_1,v,w_3,z\} \subseteq N(u) \cap N(w_2)$, contradicting $\lambda = 3$. Thus, $G[N(w_1) \cap N(w_2)] \not\cong P_3$.
\end{proof}

It should be noted that \cref{thm:P3} is a special case of \cref{thm:bipartite}, found later in the paper.

Naturally, there are a variety of graphs to sum with $P_3$ to see if it is a possible USNS for some edge-regular graph. There is a partial result for $P_3 + H$ where $H$ is an arbitrary graph.

\begin{theorem}\label{thm:P3H}
Suppose $G \in ER(n,d,\lambda)$ with a $P_3 + H$ USNS for some graph $H$. If for two adjacent vertices $u,v \in V(G)$, the $P_3$ in $G[N(u) \cap N(v)]$ consists of the vertices $w_1, w_2, w_3$, with $w_1 \sim w_2$, then $u,v \in N(w_1) \cap N(w_2)$ must be vertices in some subgraph $H$ of $G[N(w_1) \cap N(w_2)]$ and $H$ contains a $P_4$ subgraph.
\end{theorem}
\begin{proof}
Let $u,v$ be adjacent vertices in $G$, and $N(u) \cap N(v) = \{w_1,w_2,w_3,h_1,\dots,h_{|H|}\}$, where $G[w_1,w_2,w_3] \cong P_3$ and $G[h_1,\dots,h_{|H|}] \cong H$ for some arbitrary graph $H$ such that $N(u) \cap N(v) \cong P_3 + H$.

By way of contradiction, let $u$ and $v$ be vertices of $P_3$ in $G[N(w_1) \cap N(w_2)]$. The third vertex of $P_3$ in $G[N(w_1) \cap N(w_2)]$ must be an element of $A(u,v)$ or $B(u,v)$. Without loss of generality, suppose the remaining vertex is $a_1 \in A(u,v)$. Then every vertex of $H$ in $G[N(w_1) \cap N(w_2)]$ must be in $X(u,v)$, as any vertex in $A(u,v)$ or $B(u,v)$ would have an adjacency to $u$ or $v$, respectively. Thus, $N(w_1) \cap N(w_2) = \{a_1,u,v,x_1,\dots,x_{|H|}\}$.

Consider the adjacent vertices $u$ and $w_2$. Notice that $\{a_1,w_1,v,w_3\} \subseteq N(u) \cap N(w_2)$, and $G[a_1,w_1,v,w_3]$ is connected. As these four vertices are part of the same component in $N(w_2) \cap N(u)$, then they cannot contain the $P_3$ component and thus are contained in the $H$ component so $H$ must contain a $P_4$.

Now consider the adjacent vertices $v$ and $w_1$. As $\{w_2,u\} \subset N(v) \cap N(w_1)$ and $w_2 \sim u$, then $w_2$ and $u$ are in the same component of $G[N(v) \cap N(w_1)]$. The only other vertices contained in $N(v) \cap N(w_1)$ are in $B(u,v)$, so $N(v) \cap N(w_1) = \{w_2,u,b_1,\dots,b_{|H| + 1}\}$. As $|H| \geq 4$ and $|P_3| = 3$, then there must be some $b_i$ adjacent to either $w_2$ or $u$. By definition of $B(u,v)$, $b_i \not\sim u$, so $w_2 \sim b_i$ for some $1 \leq i \leq |H| + 1$. But then $N(w_1) \cap N(w_2)$ contains $b_i$, a contradiction. Thus, $u$ and $v$ cannot be vertices of $P_3$ in $G[N(w_1) \cap N(w_2)]$.
\end{proof}

A natural corollary follows from the above theorem to forbid a union of isolated vertices with $P_3$.

\begin{corollary}\label{cor:P3Kl}
    If $G \in ER(n,d,3 + \ell)$ with a USNS, then the USNS $\ncong P_3 + \ell K_1$, where $\ell \geq 1$.
\end{corollary}
\begin{proof}
    By way of contradiction, let $u,v$ be adjacent vertices in $G$ such that $N(u) \cap N(v) = \{w_1,\dots,w_{\ell+3}\}$ where $G[w_1,w_2,w_3] \cong P_3$ and $G[w_4,\dots,w_{\ell+3}] \cong \ell K_1$. Let $\ell K_1 = H$. Then $G[N(u) \cap N(v)] \cong P_3 + H$.

    By $\cref{thm:P3H}$, $u$ and $v$ are elements of $H$ in $N(w_1) \cap N(w_2)$. But $H$ has no edges and $u \sim v$, a contradiction.
\end{proof}

Since $P_3$ is a forbidden USNS, it is natural to ask if longer paths are also forbidden. While the proof cases are more numerous, the below theorem proves that $P_4$, like $P_3$, is USNS-forbidden.

\begin{theorem}\label{thm:P4}
    If $G \in ER(n,d,4)$ with a USNS, then the USNS $\ncong P_4$.
\end{theorem}
\begin{proof}
    Suppose for contradiction $\exists$ $G \in ER(n,d,4)$ with USNS $\cong P_4$. Let $u \sim v$, and let $N(u) \cap N(v) = \{w_1,w_2,w_3,w_4\}$, where $G[N(u) \cap N(v)] \cong P_4$ with endpoints $w_1$ and $w_4$ and $w_1 \sim w_2$. $G[N(w_1) \cap N(w_2)] \cong P_4$, as $G$ has a $P_4$ USNS.

    \textbf{Case 1.} $N(w_1) \cap N(w_2) = \{a_1,u,v,b_1\}$, such that $G[N(w_1) \cap N(w_2)] \cong P_4$ having endpoints $a_1$ and $b_1$, with $a_1 \in A(u,v)$ and $b_1 \in B(u,v)$. See \cref{fig:P4Case1} for reference.

    \begin{figure}[h]
        \centering
        \includegraphics[width=.4\paperwidth]{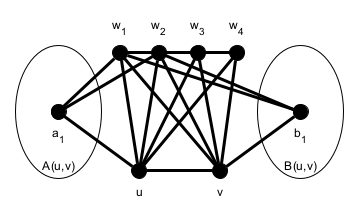}
        \caption{Beginning of case 1 in the proof of \cref{thm:P4}}
        \label{fig:P4Case1}
    \end{figure}
    
    Consider the vertices $u$ and $w_1$, which are adjacent by assumption. As vertices $u$ and $w_1$ have common neighbors $a_1, w_2,$ and $v$, then there must exist another vertex in their shared neighborhood. As $w_1$ is not adjacent to $w_3$ and $w_4$, and $u$ is only adjacent to $v$, the $w_i$ vertices, and vertices in $A(u,v)$, then the 4th vertex in this common neighborhood must be some $a_2 \in A(u,v)$. As $a_2 \in A(u,v)$, then $a_2 \not\sim v$. Since $w_1$ and $w_2$ cannot be both adjacent to $a_2$ (otherwise $a_2$ would be contained in $N(w_1) \cap N(w_2)$), then $a_2 \not\sim w_2$. So $a_2 \sim a_1$, and $G[N(u) \cap N(w_1)] \cong P_4$ with endpoints $a_2$ and $v$.

    Now consider adjacent vertices $u$ and $w_2$. $N(u) \cap N(w_2) = \{a_1,w_1,v,w_3\}$ is completely determined from previous assumptions. As $G[N(u) \cap N(w_2)] \cong P_4$, then this must have endpoints $w_3$ and $a_1$, so $w_3 \nsim a_1$.

    Now consider the adjacent vertices $v$ and $w_1$. Then $N(v) \cap N(w_1) = \{b_2,b_1,w_2,u\}$, where $b_2 \in B(u,v)$. Using similar logic to how $N(u) \cap N(w_1)$ was constructed, then we conclude that $G[N(v) \cap N(w_1)] \cong P_4$ having endpoints $b_2$ and $u$, with $b_2 \nsim w_2$ and $b_2 \sim b_1$.

    Now consider the adjacent vertices $v$ and $w_2$. $N(v) \cap N(w_2) = \{b_1,w_1,u,w_3\}$ is completely determined from previous assumptions. As $G[N(v) \cap N(w_2)] \cong P_4$, then this must have endpoints $w_3$ and $b_1$, so $w_3 \nsim b_1$.

    Lastly, consider the adjacent vertices $w_2$ and $w_3$. As $\{u,v\} \in N(w_2) \cap N(w_3)$, $\exists$ $z \in \{N(w_2) \cap N(w_3)\} \setminus \{u,v\}$ such that $z \in A(u,v)$ or $z \in B(u,v)$. As $w_2 \nsim a_2$ and $w_2 \nsim b_2$ (from $N(u) \cap N(w_2)$ and $N(v) \cap N(w_2)$, respectively), then $z \neq a_2$ and $z \neq b_2$. As $w_3 \nsim a_1$ and $w_3 \nsim b_1$ (implied from $N(u) \cap N(w_2)$ and $N(v) \cap N(w_2)$, respectively), then $z \neq a_1$ and $z \neq b_1$. Without loss of generality, say $z \in A(u,v)$. Then $N(u) \cap N(w_2)$ contains $z \in A(u,v)\setminus {a_1}$, a contradiction. Thus, $N(w_1) \cap N(w_2) \neq \{a_1,u,v,b_1\}$.

    \textbf{Case 2.} $N(w_1) \cap N(w_2) = \{v,u,a_1,x_1\}$, where $a_1 \in A(u,v)$ and $x_1 \in X(u,v)$. By assumption, $u \sim a_1$, $u \nsim x_1$, and $v \nsim x_1$, so $v$ and $x_1$ are endpoints of $G[N(w_1) \cap N(w_2)]$.

    Consider adjacent vertices $w_1$ and $v$. Then $N(w_1) \cap N(v) = \{u,w_2,b_2,b_3\}$ for some $b_2,b_3 \in B(u,v)$. This follows from the facts that $v$ has no neighbors in $A(u,v) \cup X(u,v)$ and $w_1$ is adjacent to no $w_j$; $j > 2$. Therefore, the two vertices in $N(w_1) \cap N(v)$ other than $u$ and $w_2$ must be in $B(u,v)$. By assumption, $u$ is not adjacent to any vertex in $B(u,v)$, so $w_2$ must be adjacent to one of $\{b_2,b_3\}$. Without loss of generality, $w_2 \sim b_2$. However, this implies $N(w_1) \cap N(w_2)$ contains $b_2$, a contradiction. So $N(w_1) \cap N(w_2) \neq \{v,u,a_1,x_1\}$.

    \textbf{Case 3.} $N(w_1) \cap N(w_2) = \{a_2,a_1,u,v\}$, where $a_1,a_2 \in A(u,v)$. By assumption, $u \sim a_2$ and $u \sim a_1$, so $u$ is not an endpoint of $G[N(w_1) \cap N(w_2)]$. $v$ must be an endpoint, as $v$ is only adjacent to $u$. Without loss of generality, say $a_2$ is an endpoint and $a_1$ is not an endpoint in $G[N(w_1) \cap N(w_2)]$. As $a_2 \sim u$, then $G[N(w_1 \cap N(w_2)] \ncong P_4$, a contradiction. So $N(w_1) \cap N(w_2) \neq \{a_2,a_1,u,v\}$. 
    
    This exhausts all possibilities for $N(w_1) \cap N(w_2)$, so $G$ cannot have $P_4$ as a USNS.    
\end{proof}

\begin{theorem}\label{thm:P5}
    Let $G \in ER(n,d,\lambda)$ with a $P_\lambda$ USNS for $\lambda \geq 5$, and let $u \sim v$ in $G$ with $N(u) \cap N(v) = \{w_1,w_2,\dots,w_\lambda\}$, where $w_1$ is an endpoint of $G[N(u) \cap N(v)]$. If $w_1 \sim w_2$, then $N(w_1) \cap N(w_2)$ contains exactly one vertex from $N(u) \setminus (N(u) \cap N(v))$ and exactly one vertex from $N(v) \setminus (N(u) \cap N(v))$.
\end{theorem}
\begin{proof}
    \textbf{Case 1.} $N(w_1) \cap N(w_2)$ contains no vertex from $A(u,v)$, without loss of generality. So $N(w_1) \cap N(w_2)$ contains $u, v$, a vertex in $B(u,v)$, and $\lambda - 3$ vertices in $X(u,v)$.

    Consider adjacent vertices $u$ and $w_1$. Then $N(u) \cap N(w_1)$ contains $v$ and $w_2$. But as $u$ is not adjacent to any vertex in the set $B(u,v)$ nor $X(u,v)$, the remainder of the vertices in this common neighborhood must be elements of $A(u,v)$. Yet there is no adjacency from these vertices in $A(u,v)$ to $v$. If any of these vertices in $A(u,v)$ were to be adjacent to $w_2$, then $N(w_1) \cap N(w_2)$ would contain a vertex from $A(u,v)$, contradicting our case assumption. As $\lambda \geq 5$, then $G[N(u) \cap N(w_1)] \ncong P_\lambda$, a contradiction.

    \textbf{Case 2.} $N(w_1) \cap N(w_2)$ contains more than one vertex from $A(u,v)$, say $m$ vertices from $A(u,v)$. Then $u$ in $G[N(w_1) \cap N(w_2)]$ has degree $m+1$. As $m \geq 2$, then $G[N(w_1) \cap N(w_2)] \ncong P_\lambda$, a contradiction.

    Thus, $N(w_1) \cap N(w_2)$ must contain exactly one vertex from $A(u,v)$ and exactly one vertex from $B(u,v)$.
\end{proof}

Despite paths not being completely settled as a family of USNS-forbidden graphs, there are other families of graphs that are. The following theorems tackle a few of these families, namely the family of complete bipartite graphs of different partition sizes, star graphs, and wheel graphs.

\begin{theorem}\label{thm:bipartite}
    If $G \in ER(n,d,m_1 + m_2)$ with a USNS, then for all $m_1 \neq m_2$, the USNS $\ncong K_{m_1,m_2}$.
\end{theorem}
\begin{proof}
    Let $u \sim v$, and $N(u) \cap N(v) = \{w_1,w_2,\dots,w_{m_1},z_1,z_2,\dots,z_{m_2}\}$, where $G[w_1,w_2,\dots,w_{m_1},z_1,z_2,\dots,z_{m_2}] \cong K_{m_1,m_2}$ with $w_1,\dots,w_{m_1}$ in one part and $z_1,\dots,z_{m_2}$ in the other part.

    \begin{figure}[h]
        \centering
        \includegraphics[width=.2\paperwidth]{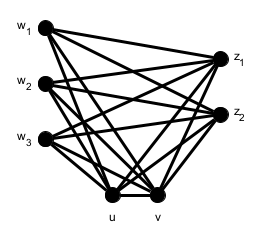}
        \caption{A $K_{3,2}$ shared neighborhood of vertices $u$ and $v$.}
        \label{fig:K32}
    \end{figure}
    
    Consider the adjacent vertices $w_1$ and $z_1$. Then without loss of generality, $N(w_1) \cap N(z_1) = \{u,v,a_1,\dots,a_{m_2-1},b_1,\dots,b_{m_1-1}\}$, where $a_1,\dots,a_{m_2-1} \in A(u,v)$ and $b_1,\dots,b_{m_1-1} \in B(u,v)$. So $G[N(w_1) \cap N(z_1)] \cong K_{m_1,m_2}$, where $v,a_1,\dots,a_{m_2-1}$ are in one part and $u,b_1,\dots,b_{m_1-1}$ are in the other part.

    Now consider the adjacent vertices $u$ and $w_1$. Then by previous assumptions, $N(u) \cap N(w_1)$ contains $\{z_1,\dots,z_{m_2},a_1,\dots,a_{m_2-1},v\}$. Further, as $\lambda = m_1 + m_2$ by assumption and $|N(u) \cap N(w_1)| \geq 2m_2$, then $m_2 \leq m_1$. By symmetry, $m_1 \leq m_2$, so $m_1 = m_2$. Thus, $K_{m_1,m_2}$ is only possible as a USNS when $m_1 = m_2$.
\end{proof}

As such, a graph such as the one in \cref{fig:K32} is a forbidden USNS, where $m_1 = 3$ and $m_2 = 2$. What immediately follows from \cref{thm:bipartite} is a fact about the \textit{star graph} $S_l$, which is a graph with one central vertex and $l-1$ vertices adjacent to it, but not to each other.

\begin{corollary}\label{cor:star}
    If $G \in ER(n,d,\ell)$ with a USNS, then for all $l \geq 3$, the USNS $\ncong S_\ell$.
\end{corollary}
\begin{proof}
    Let $m_1 = 1$ and $m_2 = \ell - 1$. Then $K_{m_1,m_2} \cong S_\ell$. So $S_\ell$ cannot be a USNS by \cref{thm:bipartite}.
\end{proof}

As noted earlier, \cref{thm:bipartite} generalizes \cref{thm:P3}, as $P_3 \cong K_{1,2}$.

This is not to suggest that complete bipartite graphs with equal part sizes are also USNS-forbidden. On the contrary, consider $K_4$, which has a $K_2 \cong K_{1,1}$ USNS.

In the following result, define the \textit{wheel graph} $W_m$ to be a connected graph on $m+1$ vertices, such that $m$ vertices induce a cycle, and the $(m+1)^{st}$ vertex is adjacent to all vertices in the cycle. Note that for $m \geq 4$, $W_m$ is not a regular graph.

\begin{theorem}\label{thm:Wheel}
    If $G \in ER(n,d,m+1)$ has a USNS, then for all $m \geq 4$, the USNS $\ncong W_m$.
\end{theorem}
\begin{proof}
    Suppose for contradiction $u \sim v$ such that $G[N(u) \cap N(v)] \cong W_m$ consisting of vertices $w_1,\dots,w_{m+1}$ such that $w_2,\dots,w_{m+1}$ are the vertices in the cycle and $w_1$ is adjacent to the vertices in the cycle.

    Consider adjacent vertices $u$ and $w_1$. $N(u) \cap N(w_1) = \{w_2,w_3,\dots,w_{m+1},v\}$. So $w_1$ is not adjacent to any vertex in $A(u,v)$.

    Similarly, $w_1$ is not adjacent to any vertex in $B(u,v)$.

    As $G[u,v,w_1] \cong K_3$ and $N(w_2) \cap N(w_3)$ contain $u,v,w_1$, then this $K_3$ is an induced subgraph of $G[N(w_2) \cap N(w_3)]$. As $m \geq 4$, one of $u,v,w_1$ must be the center of this wheel.

    If $u$ is the center, the other $m-2$ vertices in $N(w_2) \cap N(w_3)$ besides $u,v, w_1$ must be in $A(u,v)$, so $w_1 \sim a_i$ for some $a_i \in A(u,v)$, a contradiction.

    If $v$ is the center, the other $m-2$ vertices in $N(w_2) \cap N(w_3)$ besides $u,v, w_1$ must be in $B(u,v)$, so $w_1 \sim b_i$ for some $b_i \in B(u,v)$, a contradiction.

    If $w_1$ is the center, then as $m-2 > 0$, $u$ and $v$ are adjacent vertices on a cycle $C_m$ in $G[N(w_2) \cap N(w_3)]$ of length $m \geq 4$ which cannot contain any $w_j$, $j > 3$ (because $w_2 \not\sim w_j$). Then there is a $P_4$ $auvb$ on $C_m$ with $a \in A(u,v)$, $b \in B(u,v)$. But then $w_1$, as the center of the wheel, is adjacent in $G$ to both $a$ and $b$, whereas either adjacency contradicts a previous inference.

    Thus, $W_m$ is not a possible USNS when $m \geq 4$.
\end{proof}

A \textit{component}-regular graph is a graph such that each component is regular. Every known USNS graph is component-regular, and every aforementioned USNS-forbidden graph is not component-regular. Is it true that every USNS graph is component-regular?

\section{Constructions of $ER(n,d,\lambda)$ with USNS}

Given graphs $G_1$ and $G_2$, the Cartesian product of $G_1$ and $G_2$ is denoted $G_1 \square G_2$. The vertex set is defined by $V(G_1 \square G_2) = V(G_1) \times V(G_2)$. The edge set is defined by, given two vertices $(u,u')$ and $(v,v') \in V(G_1 \square G_2)$, $(u,u') \sim (v,v')$ if and only if either $u = v$ and $u' \sim v'$ (in $G_2$) or $u \sim v$ (in $G_1$) and $u' = v'$. 

It was shown in \cite{Glorioso} that if $G_1 \in ER(n_1,d_1,\lambda)$ and $G_2 \in ER(n_2,d_2,\lambda)$, then $G_1 \square G_2 \in ER(n_1n_2,d_1 + d_2,\lambda)$. However, it is rare that the Cartesian product of two edge-regular graphs that each have a USNS will have a USNS.

\begin{theorem}\label{thm:CartesianUSNS}
    Let $n_1,n_2 \geq 1$ and $d_1, d_2 \geq 0$. If $G_1 \in ER(n_1,d_1,\lambda)$ with a USNS $\cong X$ and $G_2 \in ER(n_2,d_2,\lambda)$ with a USNS $\cong Y$, then $G_1 \square G_2 \in ER(n_1n_2,d_1 + d_2,\lambda)$ has a USNS if and only if $X \cong Y$, in which case the USNS of $G_1 \square G_2$ is $X (\cong Y)$.
\end{theorem}
\begin{proof}
    Let $G_1 \in ER(n_1,d_1,\lambda)$ with USNS $\cong X$ and $G_2 \in ER(n_2,d_2,\lambda)$ with USNS $\cong Y$.

    By assumption, $G_1 \square G_2 \in ER(n_1n_2,d_1 + d_2,\lambda)$ has a USNS. Suppose $(u,v) \sim (x,y)$ in $G_1 \square G_2$. Then by the definition of the Cartesian product, either $u = x$ in $G_1$ and $v \sim y$ in $G_2$ or $u \sim x$ in $G_1$ and $v = y$ in $G_2$.

    If $u = x$ in $G_1$ and $v \sim y$ in $G_2$, then $N_{G_1 \square G_2}((u,v),(x,y)) = \{(u,z) | z \in N_{G_2}(v) \cap N_{G_2}(y)\}$ which induces, in $G_1 \square G_2$, a graph isomorphic to $Y$.

    Similarly, if $u \sim x$ in $G_1$ and $v = y$ in $G_2$, then $N_{G_1 \square G_2}((u,v),(x,y))$ induces, in $G_1 \square G_2$, a graph isomorphic to $X$. However, $G_1 \square G_2$ has a USNS, by assumption. Thus, $X \cong Y$.

    In the other direction, $X \cong Y$ by assumption. By the definition of the Cartesian product, if $(u,v) \sim (x,y)$ in $G_1 \square G_2$, then either $u = x$ in $G_1$ and $v \sim y$ in $G_2$ or $u \sim x$ in $G_1$ and $v = y$ in $G_2$. Similar to the previous direction, the intersection of open neighborhoods in $G_1 \square G_2$ yields a graph isomorphic to either $X$ or $Y$. As $X \cong Y$, then $G_1 \square G_2$ has a USNS.
\end{proof}

The tensor product of $G_1$ and $G_2$ is denoted $G_1 \otimes G_2$. The vertex set is $V(G_1 \otimes G_2) = V(G_1) \times V(G_2)$. The edge set is defined by, given two vertices $(u,u')$ and $(v,v') \in V(G_1 \otimes G_2)$, $(u,u') \sim (v,v')$ if and only if $u \sim v$ in $G_1$ and $u' \sim v'$ in $G_2$. By previous work in \cite{Glorioso}, if $G_1 \in ER(n_1,d_1,\lambda_1)$ and $G_2 \in ER(n_2,d_2,\lambda_2)$, then $(G_1 \otimes G_2) \in ER(n_1n_2,d_1d_2,\lambda_1\lambda_2)$. The following theorem extends the work in \cite{Glorioso} to include the preservation and structure of the USNS in $G_1 \otimes G_2$.

\begin{theorem}\label{thm:TensorUSNS}
    If $G_1 \in ER(n_1,d_1,\lambda_1)$ with a USNS $\cong H_1$ and $G_2 \in ER(n_2,d_2,\lambda_2)$ with a USNS $\cong H_2$, then $(G_1 \otimes G_2) \in ER(n_1n_2,d_1d_2,\lambda_1\lambda_2)$ with a USNS $\cong H_1 \otimes H_2$.
\end{theorem}
\begin{proof}
    Suppose that $(u,v) \sim (x,y)$ in $G_1 \otimes G_2$. Then $(s,t) \in N_{G_1 \otimes G_2}(u,v) \cap N_{G_1 \otimes G_2}(x,y)$ if and only if $u \sim s$, $x \sim s$ in $G_1$ and $v \sim t$, $y \sim t$ in $G_2$. Thus, $N_{G_1 \otimes G_2}(u,v) \cap N_{G_1 \otimes G_2}(x,y) = (N_{G_1}(u) \cap N_{G_1}(x)) \times (N_{G_1}(v) \cap N_{G_1}(y)) = V(H_1) \times V(H_2)$, and this set induces $H_1 \otimes H_2$ in $G_1 \otimes G_2$.
\end{proof}

For example, $K_n \otimes K_m \cong K_{m,m,\dots,m} \setminus \{(n-1)$-factor\}, an $n$-partite graph with uniform part size $m$ where the $(n-1)$-factor are the column edges when the vertices are arranged in a $n \times m$ matrix. The USNS of $(K_n \otimes K_m) \cong K_{n-2} \otimes K_{m-2} \cong K_{m-2,m-2,\dots,m-2} \setminus \{(n-3)$-factor\}, an $(n-2)$-partite graph with uniform part size $m-2$, where the $(n-3)$-factor are column edges when the vertices are arranged in a $(n-2) \times (m-2)$ matrix.

Using the above example, given $G \in ER(n,d,\lambda)$ with USNS $\cong H$ where $|H| = \lambda$, then $K_3 \otimes G$ yields a USNS of $nK_1$. In other words, the tensor product of an edge-regular graph $G$ with some USNS and a $K_3$ removes all of the edges of the USNS of $G$ as a new USNS.

Another example: $G_1 \otimes G_2$, where $G_1 \in ER(n,d,\lambda)$ and $G_2$ is a triangle-free regular graph, has an empty graph USNS.

Another useful graph construction for edge-regular graphs is the shadow of a graph. Enlarging the definition in \cite{Shadow}, given a graph $G$, define $D_m(G)$ to be the $m^{th}$ \textit{shadow graph} of $G$, by $V(D_m(G)) = \{v^i_j | i \in [m]; j \in [n]\}$, given that $V(G) = \{v_1,\dots,v_n\}$; the vertices $v^i_j$ and $v^k_l$ are adjacent in $D_m(G)$ if $v_j \sim v_l$ in $G$ for $j,l \in [n]$ and $i,k \in [m]$. See \cref{fig:D3K3} for an example.

\begin{theorem}\label{thm:ShadowUSNS}
     If $G \in ER(n,d,\lambda)$ with a USNS $\cong H$, then $D_m(G) \in ER(mn,md,m\lambda)$ with a USNS $\cong D_m(H)$.
\end{theorem}
\begin{proof}
    Let $G \in ER(n,d,\lambda)$. Then by construction the $m^{th}$ shadow of $G$ contains $m$ copies of every vertex of $G$, so $|D_m(G)| = mn$.

    Now suppose $N_{G}(v_i) = \{u_1,\dots,u_d\}$. Then $v^k_i$ is adjacent to each of $\{u^1_1,\dots,u^1_d,u^2_1,\dots,u^2_d,\dots,u^m_1,\dots,u^m_d\}$ for $k \in [m]$. So $D_m(G)$ is regular of degree $md$.

    Using similar logic, say $v_i \sim v_j$ in $G$ such that $N(v_i) \cap N(v_j) = \{u_1,\dots,u_\lambda\}$. Then $N(v^k_i) \cap N(v^l_j) = \{u^\alpha_\beta | \alpha \in [m]; \beta \in [\lambda]\}$ for $k,l \in [m]$. Thus, every pair of adjacent vertices in $D_m(G)$ share exactly $m\lambda$ vertices.

    Further, as $G[\{v_1,\dots,v_\lambda\}] \cong H$, then $N(v^k_i) \cap N(v^l_j)$ contains exactly $m$ copies of $H$, one in each shadow. The edge set among these $m$ copies of $H$ are as defined in the $m^{th}$ shadow graph. Thus, $D_m(G)$ has a USNS $\cong D_m(H)$.
\end{proof}

Iteration of a USNS with the shadow graph function allows for additional infinite families of USNS.

\begin{theorem}\label{thm:ShadowChain}
    $D_q(D_m(G)) \cong D_{qm}(G)$ for integers $q,m \geq 2$.
\end{theorem}
\begin{proof}
    Suppose $V(G) = \{v_1,\dots,v_n\}$ and $V(D_m(G)) = \{v^i_j | i = 1,\dots,m; j = 1,\dots,n\}$. Then $V(D_q(D_m(G))) = \{v^{i,k}_j | i = 1,\dots,m; j = 1,\dots,n; k = 1,\dots,q\}$.

    Alternatively, in $V(D_{qm}(G)) = \{v_1,\dots,v_{qm}\}$, relabel the vertices such that the $i^{th}$ vertex in the $q^{th}$ copy of the $m^{th}$ copy of the vertices is denoted $v^{m,q}_i$ for $i = 1,\dots,n$. So $\{v_{qm}\} = V(D_q(D_m(G)))$.

    In both cases, $v^{m_1,q_1}_i \sim v^{m_2,q_2}_j$ if $v_i \sim v_j$ in $G$ for $i \neq j; 1 \leq i,j \leq n; 1 \leq m_1,m_2 \leq m; 1 \leq q_1,q_2 \leq q$. Then $E(D_q(D_m(G))) = E(D_{qm}(G))$. So $D_q(D_m(G)) \cong D_{qm}(G)$ for all $q,m \geq 2$.
\end{proof}

As such, a chain of shadow graph functions can be processed as a single shadow graph function.

For example, $D_m(K_n) \cong K_{m,m,\dots,m} \cong T_{mn,n} \in ER(mn,m(n-1),m(n-2))$, a complete $n$-partite graph with uniform partition size $m$, commonly known as a (regular) \textit{Tur\'an graph}. $D_3(K_3) \cong T_{9,3}$ is shown in \cref{fig:D3K3}. So the USNS of $D_m(K_n)$ is $D_m(K_{n-2}) \cong K_{m,m,\dots,m} \cong T_{m(n-2),n-2}$, the Tur\'an graph on $m(n-2)$ vertices with partition size $m$ and $n-2$ parts.

\begin{figure}[h]
    \centering
    \includegraphics[width=.2\paperwidth]{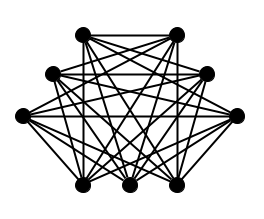}
    \caption{$D_3(K_3) \cong T_{9,3}$ with USNS of $D_3(K_1) \cong T_{3,1} = \overline{K_3}$}
    \label{fig:D3K3}
\end{figure}

As stated in the preliminaries of the paper, not all edge-regular graphs have a connected USNS. Consider $\mathcal{P}$, the Petersen graph; $\mathcal{P} \in ER(10,3,0)$.

The complement of the Petersen graph, however, is the interesting case. This graph is also already known to be edge-regular, as discussed in the $d = \lambda + 3$ case of \cite{ExtremalER}. $\mathcal{P}^C \in ER(10,6,3)$ with a USNS of $K_2 + K_1$.

\section{Conway's 99-graph Problem}

A strongly regular graph $SR(n,d,\lambda,\mu)$ is a graph in $ER(n,d,\lambda)$ such that every pair of non-adjacent vertices share exactly $\mu$ common neighbors. Conway's \textit{99-graph problem} is an open problem that asks about the existence of a strongly regular graph on parameters $SR(99,14,1,2)$ \cite{Conway}. There is some work that can be done with edge-regular graphs and the Cartesian product to show the non-existence of the 99-graph strictly using Cartesian products. 

First, define a graph $G$ as a \textit{regular clique assembly}, denoted $RCA(n,d,k)$, as a graph on $n$ vertices, regular of degree $d$, and $k = \omega(G)$ (the clique number of $G$). An RCA graph $G$ has three distinct properties: $\omega(G) \geq 2$, every maximal clique of $G$ is maximum, and each edge of $G$ is in exactly one maximum clique of $G$ \cite{RCA}.

If the 99-graph $G$ exists, then it is necessarily an edge-regular graph in $ER(99,14,1)$. This is equivalent to a regular clique assembly on the parameters $RCA(99,14,3)$ by Corollary 1 in \cite{RCA}. The idea here is to try to construct $RCA(99,14,3)$ using a Cartesian product of two graphs $G_1$ and $G_2$, and to show that there is no such combination using already proven facts about RCA graphs.

\begin{theorem}\label{thm:ConwayCartesian}
    If Conway's 99-graph exists, then it cannot be constructed with the Cartesian product of two RCA graphs.
\end{theorem}
\begin{proof}
Suppose $G_1 \in RCA(n_1,d_1,3)$ and $G_2 \in RCA(n_2,d_2,3)$. Then $G_1 \square G_2 \in RCA(n_1n_2,d_1 + d_2,3)$ by \cref{thm:CartesianUSNS}. Using the necessary conditions of $RCA$ graphs stated in Proposition 1 of \cite{RCA}, $2 \mid d_1$ and $2 \mid d_2$. There are only two options for $n_1$ and $n_2$, namely the pairs $\{33,3\}$ and $\{11,9\}$.

Let $n_1 = 3$ and $n_2 = 33$. As with all regular graphs, $n > d$, so $G_1$ must have degree 2. As such, $G_1 \in RCA(3,2,3) = ER(3,2,1) \cong K_3$. Then $G_2 \in RCA(33,12,3) = ER(33,12,1)$. By Corollary 10 of \cite{RCA}, $ER(33,12,1) = \emptyset$. So $\{33,3\}$ is not a possible pair of vertices of $G_1$ and $G_2$.

Let $n_1 = 9$ and $n_2 = 11$. Then for $G_1$ the only possible $d_2$ are $\{2,4,6,8\}$ since $n_1 = 9 > d_1$.

If $G_1 \in RCA(9,2,3)$, then $G_2 \in RCA(11,12,3)$, impossible as $n_2 < d_2$. If $G_1 \in RCA(9,4,3)$, then $G_2 \in RCA(11,10,3) = ER(11,10,1)$. Given that $n_2 = d_2 + 1$, then $G_2$ would need to be $K_{11}$, of which $\lambda = 9 \neq 1$, so $RCA(11,10,3) = ER(11,10,1) = \emptyset$. If $G_1 \in RCA(9,6,3)$, then $G_2 \in RCA(11,8,3) = ER(11,8,1)$. By Proposition 5 of \cite{RCA}, $3 \nmid nd = 88$, so $ER(11,8,1) = \emptyset$.

Finally, if $G_1 \in RCA(9,8,3) = ER(9,8,1)$, then as $n_1 = d_1 + 1$, $G_1$ is $K_9$. Yet $\{K_9\} = ER(9,8,7)$, so $ER(9,8,1) = \emptyset$.

Thus, if the 99-graph can be constructed, it cannot be done with the Cartesian product of two graphs applied to RCAs.
\end{proof}

Using similar logic, it is straightforward to show that the Tensor product of two graphs cannot yield Conway's 99-graph.

\begin{theorem}\label{thm:ConwayTensor}
    If Conway's 99-graph exists, then it cannot be constructed with the Tensor product of edge-regular graphs.
\end{theorem}
\begin{proof}
    Suppose $G_1 \in ER(n_1,d_1,\lambda_1)$ and $G_2 \in ER(n_2,d_2,\lambda_2)$ such that $G_1 \otimes G_2 \in ER(99,14,1)$. By \cref{thm:TensorUSNS}, $n_1n_2 = 99$, $d_1d_2 = 14$, and $\lambda_1\lambda_2 = 1$. Thus, $\lambda_1 = \lambda_2 = 1$. Further, $d_1d_2 = 1 \cdot 14$ or $d_1d_2 = 2 \cdot 7$.

    Suppose $d_1d_2 = 1 \cdot 14$ and without loss of generality, $d_1 = 1$. Then $\lambda_1 = 1 = d_1$, a contradiction as $d > \lambda$ for all edge-regular graphs. Thus, $d_1d_2 \neq 1 \cdot 14$.

    Suppose $d_1d_2 = 2 \cdot 7$ and without loss of generality, $d_1 = 2$. Then as $\lambda_1 = 1$ and $d_1 = 2$, $n_1 = 3$. So $n_2 = 33$, $d_2 = 7$, and $\lambda_2 = 1$. An edge-regular graph $ER(33,7,1) = RCA(33,7,3)$ by Corollary 1 in \cite{RCA}. Yet $RCA(33,7,3)$ does not exist by Proposition 1 in \cite{RCA}, as $k-1 = 2 \nmid 7$. So $d_1d_2 \neq 2 \cdot 7$.

    Thus, $G_1 \otimes G_2 \notin ER(99,14,1)$. Conway's 99-graph cannot be constructed using a Tensor product of edge-regular graphs.
\end{proof}

\newpage

\bibliography{references}
\bibliographystyle{IEEEtran}

\end{document}